\documentclass[preprint,1p]{elsarticle}
 \usepackage{amssymb}
 \usepackage{amsmath}
 \usepackage{amsthm}
\newtheorem{theorem}{Theorem}[section]
\newtheorem{corollary}[theorem]{Corollary}
\newtheorem{lemma}[theorem]{Lemma}
\newtheorem{remark}[theorem]{Remark}

\journal{Journal of Differential Equations}

\begin{document}

\begin{frontmatter}
\title{Blow-up and global existence for a general class of nonlocal nonlinear coupled wave equations}
\author{N. Duruk$^1$}
\ead{nilayduruk@su.sabanciuniv.edu}
\author{H. A. Erbay$^2$\corref{cor1}}
\ead{erbay@isikun.edu.tr}
\author{A. Erkip$^1$}
\ead{albert@sabanciuniv.edu}
\cortext[cor1]{Corresponding author. Tel:
+90 216 528 7115 Fax: +90 216 712 1474}

\address{$^1$ Faculty of Engineering and Natural Sciences, Sabanci University,
        Tuzla 34956,  Istanbul,    Turkey}

\address{$^2$ Department of Mathematics, Isik  University, Sile 34980, Istanbul, Turkey}

\begin{abstract}
We study the initial-value problem for a general class of nonlinear nonlocal coupled wave equations.
The problem involves convolution operators  with kernel functions whose Fourier transforms are
nonnegative. Some well-known examples of nonlinear wave equations, such as coupled Boussinesq-type equations
arising in elasticity and in quasi-continuum approximation of dense lattices,
follow from the present model for suitable choices of the kernel functions.  We establish local existence and
sufficient conditions for finite time blow-up and as well as global existence of solutions of the problem.
\end{abstract}

 \end{frontmatter}

\setcounter{equation}{0}
\section{Introduction}
\noindent
In this article we focus on blow-up and global existence of solutions  to the nonlocal nonlinear Cauchy problem
\begin{eqnarray}
    && u_{1tt}=(\beta_{1} \ast (u_{1}+g_{1}(u_{1},u_{2})))_{xx}, ~~~~~x\in {\Bbb R}, ~~~~t>0\label{cp1} \\
    && u_{2tt}=(\beta_{2} \ast (u_{2}+g_{2}(u_{1},u_{2})))_{xx},  ~~~~~x\in {\Bbb R}, ~~~~t>0 \label{cp2}\\
    && u_{1}(x,0)=\varphi_{1}(x),~~~u_{1t}(x,0)=\psi_{1}(x)\label{iv1}\\
    && u_{2}(x,0)=\varphi_{2}(x),~~~u_{2t}(x,0)=\psi_{2}(x)\label{iv2}.
\end{eqnarray}
Here $u_{i}=u_{i}(x,t)$ $(i=1,2)$, the subscripts $x, t$ denote partial derivatives, the symbol $~\ast$
denotes convolution in the spatial domain
\begin{displaymath}
    \beta \ast v= \int_{\Bbb R} \beta(x-y)v(y)\mbox{d}y.
\end{displaymath}
We assume that the nonlinear functions $g_{i}(u_{1},u_{2})$ $(i=1,2)$
satisfy the exactness condition
\begin{equation}
    {{\partial g_{1}}\over {\partial u_{2}}}={{\partial g_{2}}\over {\partial u_{1}}} \label{exact}
\end{equation}
or equivalently there exists a function $G(u_{1},u_{2})$ satisfying
\begin{equation}
    g_{i}={{\partial G}\over {\partial u_{i}}} ~~~~~(i=1,2). \label{gg}
\end{equation}
We assume that the kernel functions $\beta_{i}(x)$  are integrable and their Fourier transforms $\widehat{\beta_{i}}(\xi)$
satisfy
\begin{equation}
    0\leq \widehat{\beta_{i}}(\xi)\leq C_{i}(1+\xi^2)^{-{r_{i}/2}}~~\mbox{for all}~\xi~~~~(i=1,2)
\end{equation}
for  some constants $C_{i}>0$. Here the exponents $r_{1},r_{2}$ are not necessarily integers.

Equations (\ref{cp1})-(\ref{cp2}) may be viewed as a natural generalization of the single equation arising in
one-dimensional nonlocal elasticity \cite{duruk1} to a coupled system of two nonlocal nonlinear equations.
As a special case, consider  $g_{i}(u_{1},u_{2})=u_{i}W^{\prime}(u_{1}^{2}+u_{2}^{2})~~~(i=1,2)$ where $W$
is a function of $u_{1}^{2}+u_{2}^{2}$ alone and the symbol $~^\prime$ denotes the derivative. Then,
(\ref{cp1})-(\ref{cp2}) may be thought of the system governing
the one-dimensional propagation of two "pure" transverse nonlinear waves in a nonlocal elastic isotropic
homogeneous medium \cite{duruk2}. Note that this choice of $g_{1}$ and $g_{2}$ will satisfy the exactness
condition (\ref{exact}) with $G(u_{1},u_{2})={1\over 2}W(u_{1}^{2}+u_{2}^{2})$. From the modelling point of view
we want to remark that, in general, the system will also contain a third equation characterizing the propagation
of a longitudinal wave. Nevertheless, with some further restrictions imposed on the form of $W$, one may
get transverse waves without a coupled longitudinal wave \cite{haddow}. We also want to note that, in the general case,
the exactness condition (\ref{exact}) is necessary in order to obtain the conservation law of Lemma \ref{energy}.

For suitable choices of the kernel functions, the system (\ref{cp1})-(\ref{cp2}) reduces to some well-known coupled
systems of nonlinear wave equations. To illustrate this we consider the exponential kernel
$\beta_{1}(x)=\beta_{2}(x)={\frac{1}{2}}e^{-|x|}$ which is the Green's function for the operator  $~1-D_{x}^{2}~$
where $~D_{x}~$ stands for the partial derivative with respect to $x$. Then, applying the operator $~1-D_{x}^{2}~$ to
both sides of equations (\ref{cp1})-(\ref{cp2}) yields the coupled improved Boussinesq equations
\begin{eqnarray}
   && u_{1tt}-u_{1xx}-u_{1xxtt}=(g_{1}(u_{1},u_{2}))_{xx}  \label{cibe1}  \\
   && u_{2tt}-u_{2xx}-u_{2xxtt}=(g_{2}(u_{1},u_{2}))_{xx}.  \label{cibe2}
\end{eqnarray}
Similarly, if the kernels $\beta_{1}(x)$ and $\beta_{2}(x)$ are chosen as the Green's function for the fourth-order
operator $~1-aD_{x}^{2}+bD_{x}^{4}~$ with positive constants $a,b$, then (\ref{cp1})-(\ref{cp2})
reduces to  the coupled higher-order Boussinesq system
\begin{eqnarray}
    && u_{1tt}-u_{1xx}-au_{1xxtt}+bu_{1xxxxtt}=(g_{1}(u_{1},u_{2}))_{xx} \label{chbe1} \\
    && u_{2tt}-u_{2xx}-au_{2xxtt}+bu_{2xxxxtt}=(g_{2}(u_{1},u_{2}))_{xx}. \label{chbe2}
\end{eqnarray}
These examples make it obvious that choosing the kernels $\beta_{i}(x)$ in (\ref{cp1})-(\ref{cp2}) as the
Green's functions of constant coefficient linear differential operators in $x$ will yield similar coupled systems
describing the bi-directional propagation of nonlinear waves in dispersive media.  Different examples of the kernel
functions used in the literature can be found in \cite{duruk1} where such kernels will give not only differential
equations but also integro-differential equations or difference-differential equations. For a survey of Korteweg-de Vries
type nonlinear nonlocal equations of hydrodynamic relevance we refer to \cite{constantin1}.

The coupled improved Boussinesq system (\ref{cibe1})-(\ref{cibe2}) has been derived to describe bi-directional wave
propagation in various contexts, for instance, in a Toda lattice model with a transversal degree of freedom
\cite{christiansen}, in a two-layered lattice model \cite{khusnutdinova} and in a diatomic lattice \cite{wattis}.
For a discussion of the classical Boussinesq system we refer to  \cite{alvarez, constantin2}.
The Cauchy problem for (\ref{cibe1})-(\ref{cibe2}) has been studied in \cite{godefroy} and recently in \cite{wang}
where both assume the exactness condition (\ref{exact}). They have established the conditions for the global existence
and finite-time blow-up of solutions in Sobolev spaces $H^{s}\times H^{s}$ for $s>1/2$.

The single component form of equations (\ref{chbe1})-(\ref{chbe2}) arises as a model for a dense  chain of particles with
elastic couplings \cite{rosenau}, for water waves with surface tension \cite{schneider} and for longitudinal waves in a
nonlocal nonlinear elastic medium \cite{duruk2}. We have proved in \cite{duruk2} that the Cauchy problem for the single
component form of (\ref{chbe1})-(\ref{chbe2}) is globally well-posed in Sobolev spaces $H^{s}$ for $s>1/2$ under certain
conditions on nonlinear term and initial data. To the best of our knowledge, the questions of global well-posedness
and finite-time blow-up of solutions for the coupled higher-order Boussinesq system (\ref{chbe1})-(\ref{chbe2}) are open
problems. In this article we shall resolve these problems by considering a closely related, but somewhat more general,
problem defined by (\ref{cp1})-(\ref{iv2})

In Section 2 we present a local existence theory   of the Cauchy problem (\ref{cp1})-(\ref{iv2}) for the case of
general kernels with $r_{1}, r_{2}\geq 2$ and initial data in  suitable Sobolev spaces.
In Section 3 we prove the energy identity and in Section 4 we discuss finite time blow-up of solutions of the
initial-value problem. Finally, in Section 5 we prove two separate results on global existence of solutions of
(\ref{cp1})-(\ref{iv2}) for two different classes of kernel functions.

In what follows $H^{s}=H^{s}({\Bbb R})$ will denote the $L^{2}$ Sobolev space on ${\Bbb R}$. For the $H^{s}$ norm
we use the Fourier transform representation
$~\left\Vert u \right\Vert _{s}^{2}=\int_{\Bbb R} (1+\xi^{2})^{s}|\hat{u}(\xi)|^{2} \mbox{d}\xi~$.
We use $~\left\Vert u \right\Vert _{\infty}~$, $~\left\Vert u \right\Vert~$
and $~\left\langle u, v\right\rangle~$ to denote the $L^{\infty}$ and $L^{2}$ norms and the inner product in $L^{2}$,
respectively.

\setcounter{equation}{0}
\section{Local Well Posedness}
\noindent
To shorten the notation  we write $f_{i}( u_{1},u_{2}) =u_{i}+g_{i}(u_{1},u_{2})$ $~(i=1,2)$.
Note that
\begin{equation}
    f_{i}={{\partial F}\over {\partial u_{i}}} ~~~~~(i=1,2)
\end{equation}
where $F(u_{1},u_{2})={1\over 2} (u_{1}^{2}+u_{2}^{2})+G(u_{1},u_{2})$.

For a vector function $U=(u_{1},u_{2})$ we define the norms
$~\left\Vert U \right\Vert_{s}=\left\Vert u_{1} \right\Vert_{s}+\left\Vert u_{2} \right\Vert_{s}~$ and
$~\left\Vert U \right\Vert_{\infty}=\left\Vert u_{1} \right\Vert_{\infty}+\left\Vert u_{2} \right\Vert_{\infty}$.
We first need vector-valued versions of  Lemma 3.1 and Lemma 3.2 in \cite{duruk1} (see also  \cite{wang,constantin3,runst}),
which concern the behavior of the nonlinear terms:
\begin{lemma}\label{lem1}
    Let $s\geq 0,$ $h\in C^{[s]+1}({\Bbb R}^{2})$ with $h(0)=0$. Then for any
    $U=(u_{1},u_{2})\in (H^{s}\cap L^{\infty })^{2}$, we have
    $h(U) \in H^{s}\cap L^{\infty}$. Moreover there is some constant $A(M)$ depending on $M$ such that for
    all $U\in (H^{s}\cap L^{\infty })^{2}$ with $\left\Vert U\right\Vert _{\infty }\leq M$
    \begin{equation*}
        {\left\Vert h(U)\right\Vert }_{s}\leq A(M){\left\Vert U\right\Vert}_{s}~.
    \end{equation*}
\end{lemma}
\begin{lemma}\label{lem2}
    Let $s\geq 0$, $h\in C^{[s]+1}({\Bbb R}^{2})$. Then for any $M>0$ there is some constant $B(M)$ such that for
    all $~U,V\in (H^{s}\cap L^{\infty })^{2}$ with
    $~\left\Vert U\right\Vert _{\infty }\leq M$, $~\left\Vert V\right\Vert _{\infty }\leq M~$ and
    $~\left\Vert U\right\Vert _{s}\leq M$, $~\left\Vert V\right\Vert_{s}\leq M~$ we have
    \begin{equation*}
    {\left\Vert h(U)-h(V)\right\Vert }_{s}\leq B(M){\left\Vert U-V\right\Vert }_{s}~~\mbox{ and  }~~
           {\left\Vert h(U)-h(V)\right\Vert }_{\infty }\leq B(M){\left\Vert U-V\right\Vert }_{\infty }~.
    \end{equation*}
\end{lemma}
The Sobolev embedding theorem implies that $~H^{s}\subset L^{\infty }$ for $~s>\frac{1}{2}$.
 Then the bounds on
$L^{\infty }$ norms in Lemma \ref{lem2}  appear unnecessary  and we get:
\begin{corollary}\label{cor3}
    Let $~s>\frac{1}{2}$, $~h\in C^{[s]+1}({\Bbb R}^{2})$. Then for any $M>0$ there is some constant $B(M)$
    such that for all $~U,V\in (H^{s})^{2}~$ with
    $~\left\Vert U\right\Vert_{s}\leq M$, $~\left\Vert V\right\Vert_{s}\leq M$
    we have
    \begin{equation*}
        {\left\Vert h(U)-h(V)\right\Vert }_{s}\leq B(M){\left\Vert U-V\right\Vert }_{s}~.
    \end{equation*}
\end{corollary}
Throughout this paper we assume  that $~f_{1},f_{2}\in C^{\infty}({\Bbb R}^{2})~$ with $~f_{1}(0)=f_{2}(0)=0~$.
In the case of  $~f_{1},f_{2}\in C^{k+1}(\mathbb{R}^{2})$, Lemma \ref{lem1} and Lemma \ref{lem2} will hold only for $s\leq k$. Thus
all the results below will hold for $s\leq k$. Note that the functions $g_{1}$ and $g_{2}$ appearing in (\ref{cp1}) and
(\ref{cp2}) will also satisfy the same assumptions as $f_{1}$ and $f_{2}$.
\begin{theorem}\label{theo4}
    Let $~s>1/2~$  and $~r_{1}, r_{2}\geq 2$. Then there is some $~T>0~$ such that the Cauchy problem (\ref{cp1})-(\ref{iv2})
    is well posed with solution $u_{1}$ and $u_{2}$ in $~C^{2}( \left[ 0,T\right],H^{s})~$ for initial data
    $~\varphi_{i}, \psi_{i} \in H^{s}$ $~~(i=1,2)$.
\end{theorem}
\begin{proof}
    We convert the problem into an $H^{s}$ valued system of ordinary differential equations
    \begin{equation*}
        \begin{array}{ll}
            u_{1t}=v_{1}, & ~~u_{1}( 0) =\varphi_{1}, \\
            u_{2t}=v_{2}, & ~~u_{2}( 0) =\varphi_{2}, \\
            v_{1t}=\beta_{1}\ast (f_{1}( u_{1},u_{2}))_{xx}, ~~&~~ v_{1}( 0) =\psi_{1}, \\
            v_{2t}=\beta_{2}\ast (f_{2}( u_{1},u_{2}))_{xx}, ~~&~~ v_{2}( 0) =\psi_{2}.
        \end{array}
    \end{equation*}%
    In order to use the standard well-posedness result \cite{ladas} for ordinary differential equations, it suffices
    to show that the right hand side is Lipschitz on     $H^{s}$.  Since $r_{i}\geq 2$ for  $i=1,2$, we have
    \begin{equation*}
        \left\vert -\xi ^{2}\widehat{\beta_{i}}(\xi)\right\vert \leq C_{i}\xi ^{2}(1+\xi ^{2})^{-r_{i}/2}\leq C_{i}.
    \end{equation*}
    Then we get
    \begin{eqnarray}
        \left\Vert \beta_{i}\ast w_{xx}\right\Vert _{s} & = &
            \left\Vert (1+\xi ^{2})^{s/2}\xi ^{2}\widehat{\beta}_{i}(\xi )\widehat{w}(\xi )\right\Vert  \nonumber \\
                                            &\leq &
            C_{i}\left\Vert (1+\xi ^{2})^{s/2}\widehat{w}(\xi )\right\Vert=C_{i}\left\Vert w\right\Vert _{s}~. \label{aa}
    \end{eqnarray}
    This implies that $~\beta_{i}\ast ( . )_{xx}~$ is a bounded linear map on $~H^{s}$. Then it follows  from
    Corollary \ref{cor3} that  $~\beta_{i}\ast (f_{i}(u_{1},u_{2}))_{xx}~$ is locally Lipschitz on $~H^{s}$
    for $~s>\frac{1}{2}$.
\end{proof}
\begin{remark}\label{extension}
    In  Theorem \ref{theo4} we have not used neither the assumption
    $~\widehat{\beta} ( \xi ) \geq 0$ nor the exactness condition (\ref{exact}); so in fact the local existence result
    holds for more general forms of kernel functions and nonlinear terms. Moreover, as in \cite{duruk1},
    for certain classes of kernel functions Theorem \ref{theo4} can be extended to the case of $H^{s}\cap L^{\infty}$ for
    $~0\leq s \leq 1/2$.
\end{remark}

The solution in Theorem \ref{theo4} can be extended to a maximal time interval of existence
$\left[ 0,T_{\max }\right) $ where  finite $ T_{\max }$ is characterized by the blow up condition
\begin{equation*}
    \limsup_{t\rightarrow T_{\max }^{-}}( \left\Vert U( t)
        \right\Vert _{s}+\left\Vert U_{t}( t) \right\Vert_{s})=\infty,
\end{equation*}
where $U_{t}=(u_{1t},u_{2t})$. Then the solution is global, i.e. $T_{\max }=\infty$ iff
\begin{equation}
    \mbox{for any} ~T<\infty, \mbox{ we have}~\limsup_{t\rightarrow T^{-}}
        ( \left\Vert U(t) \right\Vert_{s}
        +\left\Vert U_{t}(t) \right\Vert_{s})<\infty ~. \label{globala}
\end{equation}
\begin{lemma}\label{lem9}
    Let $~s>1/2~$, $~r_{1}, r_{2}\geq 2$ and let $U$ be the solution of the Cauchy problem
    (\ref{cp1})-(\ref{iv2}). Then there is a global solution if and only if
    \begin{equation}
        \mbox{for any }~T<\infty, \mbox{ we have}~\limsup_{t\rightarrow T^{-}}
            \left\Vert U( t) \right\Vert_{\infty }<\infty ~. \label{globalb}
    \end{equation}
\end{lemma}
\begin{proof}
    We will show that the two conditions (\ref{globala}) and (\ref{globalb}) are equivalent. First assume that
    (\ref{globala}) holds. By the Sobolev imbedding theorem,
    $\left\Vert U( t) \right\Vert _{\infty }\leq C \left\Vert U( t)\right\Vert_{s} $ for $~s>1/2~$ so (\ref{globalb}) holds.
    Conversely, assume that the solution exists for $t\in \left[ 0,T\right)$. Then
    $~M=\limsup_{t\rightarrow T^{-}}\left\Vert U( t) \right\Vert_{\infty } ~$ is finite and
    $\left\Vert U( t) \right\Vert _{\infty }\leq M$ for all $0\leq t\leq T$.
    If we integrate (\ref{cp1})-(\ref{cp2}) twice and compute the resulting double integral as an iterated integral,
    we get, for   $i=1,2$,
    \begin{eqnarray}
        u_{i}( x,t) &=&
            \varphi_{i} ( x) +t\psi_{i} ( x)
            +\int_{0}^{t}( t-\tau ) (\beta_{i}\ast f_{i} (u_{1},u_{2}))_{xx}
             (x,\tau) d\tau \label{uint} \\
        u_{it}( x,t) &=&
            \psi_{i} ( x) +\int_{0}^{t}(\beta_{i}\ast f_{i} (u_{1},u_{2}))_{xx}
             (x,\tau) d\tau  ~. \label{utint}
    \end{eqnarray}
    So, \ for all $t\in \left[ 0,T\right) $ and $i=1,2$
    \begin{eqnarray*}
        \left\Vert u_{i}( t) \right\Vert _{s} &\leq &
            \left\Vert \varphi_{i} \right\Vert _{s}+T\left\Vert \psi_{i} \right\Vert_{s}
            +T\int_{0}^{t}\left\Vert (\beta_{i}\ast f_{i} (u_{1},u_{2}))_{xx}(\tau)
             \right\Vert _{s}d\tau , \\
        \left\Vert u_{it}( t) \right\Vert _{s} &\leq &
            \left\Vert \psi_{i}\right\Vert _{s}+\int_{0}^{t}
            \left\Vert (\beta_{i}\ast f_{i} (u_{1},u_{2}))_{xx}(\tau) \right\Vert_{s}d\tau ~.
    \end{eqnarray*}
    Note that
    $\left\Vert (\beta_{i}\ast f_{i} (u_{1},u_{2}))_{xx} (\tau)\right\Vert_{s}
        \leq C_{i} \Vert  f_{i}( u_{1},u_{2}) (\tau) \Vert_{s}
        \leq C_{i} A_{i}( M) \Vert u_{i}(\tau) \Vert_{s}$
    where the first inequality follows from (\ref{aa}) and the second  from Lemma \ref{lem1}.
    Adding the four inequalities     we get
    \begin{eqnarray*}
      \left\Vert U( t) \right\Vert_{s}
            +\left\Vert U_{t}( t)\right\Vert_{s} &\leq &
            \left\Vert \varphi_{1} \right\Vert _{s}
            +\left\Vert \varphi_{2} \right\Vert _{s}
            +( T+1)(\left\Vert \psi_{1}\right\Vert_{s}+\left\Vert \psi_{2} \right\Vert_{s}) \\
            &+& ( T+1) C A( M)\int_{0}^{t}\left\Vert U( \tau ) \right\Vert_{s}d\tau,
    \end{eqnarray*}
    where $C=\mbox{max}(C_{1},C_{2})$ and $A(M)=\mbox{max}(A_{1}(M),A_{2}(M))$.  Gronwall's Lemma implies
    \begin{equation*}
            \left\Vert U( t) \right\Vert_{s}
            +\left\Vert U_{t}( t)\right\Vert_{s}\leq
            [ \left\Vert \varphi_{1} \right\Vert _{s}
            +\left\Vert \varphi_{2} \right\Vert _{s}
            +( T+1)(\left\Vert \psi_{1}\right\Vert_{s}+\left\Vert \psi_{2} \right\Vert_{s}) ]
            e^{( T+1) C A( M)T}
    \end{equation*}
    for all $t\in [0,T)$ and consequently
    \begin{equation*}
       \limsup_{t\rightarrow T^{-}}~(\left\Vert U(t) \right\Vert_{s}
            +\left\Vert U_{t}(t) \right\Vert_{s})<\infty ~.
    \end{equation*}
\end{proof}

\setcounter{equation}{0}
\section{Conservation of Energy}
\noindent
In the rest of the study we will assume that $\widehat{\beta}_{i}(\xi )$ has only isolated zeros for $i=1,2$.
Let $P_{i}$ be operator defined by
$P_{i}w=\mathcal{F}^{-1}(\left\vert \xi \right\vert^{-1}(\widehat{\beta}_{i}(\xi))^{-1/2}\widehat{w}( \xi ))$
with the inverse Fourier transform $\mathcal{F}^{-1}$. Note that although $P_{i}$ may fail to be a bounded operator, its
inverse
$P_{i}^{-1}: H^{s+\frac{r_{i}}{2}}\rightarrow H^{s}$ is bounded and one-to-one for $s\geq 0.$ Then $P_{i}$ is
well defined with $domain (P_{i})=range (P_{i}^{-1})$.
Clearly, $P_{i}^{-2}w=-(\beta _{i}\ast w)_{xx}=-\beta _{i}\ast w_{xx}$.
\begin{lemma}\label{ek}
    Let $~s>1/2~$  and $~r_{1}, r_{2}\geq 2$. Suppose the solution of the Cauchy problem
    (\ref{cp1})-(\ref{iv2}) exists with  $u_{1}$ and $u_{2}$
    in $C^{2}(\left[ 0,T\right), H^{s}\cap L^{\infty })$ $\ $\ for some $s>1/2$.
    If $P_{1}\psi_{1},P_{2}\psi_{2}\in L^{2}$, then $P_{1}u_{1t},P_{2}u_{2t}\in C^{1}(\left[ 0,T\right) ,L^{2})$.
    If moreover $P_{1}\varphi_{1},P_{2}\varphi_{2}\in L^{2}$, then
    $P_{1}u_{1},P_{2}u_{2}\in C^{2}(\left[ 0,T\right),L^{2})$.
\end{lemma}
\begin{proof}
    It follows from  (\ref{utint}) that for $i=1,2$
    \begin{eqnarray*}
        P_{i}u_{it}(x, t)= P_{i}\psi_{i} (x)
            -\int_{0}^{t} ( P_{i}^{-1} f_{i}(u_{1},u_{2}))(x,\tau) d\tau ~.
    \end{eqnarray*}
    It is clear from Lemma \ref{lem1} that   $f_{i}(u_{1},u_{2}) \in H^{s}$. Also,
    $P_{i}^{-1}w=\mathcal{F}^{-1}(\left\vert \xi \right\vert (\widehat{\beta}_{i}(\xi))^{1/2}\widehat{w}(\xi))$
    thus $P_{i}^{-1}( f_{i}(u_{1},u_{2}) ) \in H^{s+\frac{r_{i}}{2}-1}\subset L^{2}$ and hence
    $P_{i}u_{it} \in L^{2}$. The continuity and differentiability of $P_{i}u_{i}$ in $t$ follows from the integral
    representation above.  With a similar approach (\ref{uint}) gives the second statement.
\end{proof}
\begin{lemma}\label{energy}
    Let $~s>1/2~$  and $~r_{1}, r_{2}\geq 2$. Suppose that $(u_{1},u_{2})$ satisfies (\ref{cp1})-(\ref{iv2})
    on some interval $\left[ 0,T\right)$. If $P_{1}\psi_{1}, P_{2}\psi_{2}\in L^{2}$  and the function $G(\varphi_{1}, \varphi_{2})$
    defined by (\ref{gg}) belongs to $L^{1}$, then for any $t\in \left[ 0,T\right)$ the energy
    \begin{eqnarray*}
        E(t) &=&
             \left\Vert P_{1}u_{1t}(t) \right\Vert^{2}
            +\left\Vert P_{2}u_{2t}(t) \right\Vert^{2}
            +\left\Vert u_{1}(t) \right\Vert^{2}
            +\left\Vert u_{2}(t) \right\Vert^{2}
            +2\int_{{\Bbb R}}G (u_{1},u_{2}) dx \\
            &=&
             \left\Vert P_{1}u_{1t}(t) \right\Vert^{2}
            +\left\Vert P_{2}u_{2t}(t) \right\Vert^{2}
            +2\int_{{\Bbb R}}F(u_{1},u_{2}) dx
    \end{eqnarray*}
    is constant in $\left[ 0,T\right)$.
\end{lemma}
\begin{proof}
     Lemma \ref{ek} says that $P_{i}u_{it}( t) \in L^{2}$ for $~i=1,2$. Equations  (\ref{cp1})-(\ref{cp2}) become
     $P_{i}^{2}u_{itt}+u_{i}+g_{i}( u_{1},u_{2}) =0$ $~(i=1,2)$. Multiplying by
     $2u_{it}$, integrating in $x$, adding the two equalities  and using Parseval's identity we obtain
     $\frac{dE}{dt}=0$.
\end{proof}

\setcounter{equation}{0}
\section{Blow Up in Finite Time}
\noindent
The following lemma will be used in the sequel to prove  blow up in finite time.
\begin{lemma}\label{blow-up} \cite{varga,levine}
    Suppose $\Phi( t)$, $t\geq 0$ is a positive, twice differentiable function satisfying
    $\Phi^{\prime \prime }\Phi-( 1+\nu ) (\Phi^{\prime })^{2}\geq 0$ where $\nu >0$.
    If $\Phi( 0) >0$ and $\Phi^{\prime }( 0) >0$, then
    $\Phi( t) \rightarrow \infty $ as $t\rightarrow t_{1}$ for some
    $t_{1}\leq \Phi( 0) /\nu \Phi^{\prime }( 0) $.
\end{lemma}
\begin{theorem}
    Let $~s>1/2~$  and $~r_{1}, r_{2}\geq 2$. Suppose that
    $P_{1}\varphi_{1}, P_{2}\varphi_{2}, P_{1}\psi_{1}, P_{2}\psi_{2} \in L^{2}$ and
    $~~G(\varphi_{1},\varphi_{2}) \in L^{1}$. If there is some
    $\nu >0$ such that
    \begin{equation*}
    u_{1} f_{1}(u_{1},u_{2})+u_{2} f_{2}(u_{1},u_{2})  \leq 2 ( 1+2\nu ) F(u_{1},u_{2}),
    \end{equation*}
    and
    \begin{equation*}
    E( 0) =\left\Vert P_{1}\psi_{1} \right\Vert^{2}+\left\Vert P_{2}\psi_{2} \right\Vert^{2}
    +2\int_{\Bbb R} F (\varphi_{1},\varphi_{2}) dx<0,
    \end{equation*}
    then the solution $(u_{1},u_{2})$ of the Cauchy problem (\ref{cp1})-(\ref{iv2}) blows up in finite time.
\end{theorem}
\begin{proof}
    Let
    \begin{equation*}
    \Phi(t)=\left\Vert P_{1}u_{1}(t)\right\Vert^{2}+\left\Vert P_{2}u_{2}(t)\right\Vert^{2}+b(t+t_{0})^2
    \end{equation*}
    for some positive $b$ and $t_{0}$ that will be specified later.
    Assume that the maximal time of existence of the solution of the Cauchy problem (\ref{cp1})-(\ref{iv2})
    is infinite. Then $P_{1}u_{1}(t),P_{1}u_{1t}(t),P_{2}u_{2}(t),P_{2}u_{2t}(t)\in L^{2}$ for
    all $t>0$; thus $\Phi(t)$ must be finite for all $t$. However, we will show below that $\Phi(t)$  blows up in
    finite time.

    We have
    \begin{eqnarray*}
        && \!\!\!\!\!\!\!\!\!\!\!\!\!\!\!\!\!\!
        \Phi ^{\prime}(t)=2\left\langle P_{1}u_{1},P_{1}u_{1t}\right\rangle
        +2\left\langle P_{2}u_{2},P_{2}u_{2t}\right\rangle+2b(t+t_{0}), \\
        && \!\!\!\!\!\!\!\!\!\!\!\!\!\!\!\!\!\!
        \Phi ^{\prime \prime }(t)=2\left\Vert P_{1}u_{1t}\right\Vert ^{2}+2\left\Vert P_{2}u_{2t}\right\Vert^{2}
            +2\left\langle P_{1}u_{1},P_{1}u_{1tt}\right\rangle+2\left\langle P_{2}u_{2},P_{2}u_{2tt}\right\rangle+2b~.
    \end{eqnarray*}
    Since
    \begin{equation*}
        \left\langle P_{i}u_{i},P_{i}u_{itt}\right\rangle = \left\langle u_{i},P_{i}^{2}u_{itt} \right\rangle
        =-\left\langle u_{i},f_{i}(u_{1},u_{2})\right\rangle, ~~~~~i=1,2,
    \end{equation*}
    and
    \begin{eqnarray*}
       -\int [u_{1} f_{1}(u_{1},u_{2})+u_{2} f_{2}(u_{1},u_{2})]dx & \geq & - 2 ( 1+2\nu ) \int F(u_{1},u_{2})dx \\
        & = &(1+2\nu)(\left\Vert P_{1}u_{1t}(t) \right\Vert^{2}
                +\left\Vert P_{2}u_{2t}(t) \right\Vert^{2}-E(0)),
    \end{eqnarray*}
    we get
    \begin{eqnarray*}
        \Phi^{\prime \prime }(t)
        & \geq &  2\left\Vert P_{1}u_{1t}\right\Vert ^{2}+2\left\Vert P_{2}u_{2t}\right\Vert^{2}+2b-2(1+2\nu )
                (E(0)-\left\Vert P_{1}u_{1t}\right\Vert ^{2}-\left\Vert P_{2}u_{2t}\right\Vert ^{2}) \\
        &= & -2(1+2\nu )E(0)+2b+4(1+\nu )(\left\Vert P_{1}u_{1t}\right\Vert^{2}+\left\Vert P_{2}u_{2t}\right\Vert ^{2}).
    \end{eqnarray*}
    By the Cauchy-Schwarz inequality we have
    \begin{eqnarray*}
        (\Phi^{\prime}(t))^{2}
        &=& 4[\left\langle P_{1}u_{1},P_{1}u_{1t}\right\rangle+\left\langle P_{2}u_{2},P_{2}u_{2t}\right\rangle+b(t+t_{0})]^{2} \\
        &\leq & 4[\left\Vert P_{1}u_{1}\right\Vert \left\Vert P_{1}u_{1t}\right\Vert
                 +\left\Vert P_{2}u_{2}\right\Vert \left\Vert P_{2}u_{2t}\right\Vert
                 +b(t+t_{0})]^{2}~.
    \end{eqnarray*}
    For the mixed terms we use the inequalities
    \begin{equation*}
        2\left\Vert P_{1}u_{1}\right\Vert \left\Vert P_{1}u_{1t}\right\Vert\left
                 \Vert P_{2}u_{2}\right\Vert \left\Vert P_{2}u_{2t}\right\Vert
            \leq  \left\Vert P_{1}u_{1}\right\Vert^{2}\left\Vert P_{2}u_{2t}\right\Vert^{2}
            + \left\Vert P_{2}u_{2}\right\Vert^{2}\left\Vert P_{1}u_{1t}\right\Vert^{2}
    \end{equation*}
    and
    \begin{equation*}
        2\left\Vert P_{i}u_{i}\right\Vert \left\Vert P_{i}u_{it}\right\Vert (t+t_{0})
            \leq  \left\Vert P_{i}u_{i}\right\Vert^{2}+\left\Vert P_{i}u_{it}\right\Vert^{2} (t+t_{0})^{2},~~~~i=1,2,
    \end{equation*}
    to obtain
    \begin{equation*}
        (\Phi^{\prime}(t))^{2} \leq 4 \Phi(t)(\left\Vert P_{1}u_{1t}\right\Vert^{2}
                                    + \left\Vert P_{2}u_{2t}\right\Vert^{2}+b)~.
    \end{equation*}
    Therefore,
    \begin{eqnarray*}
     && \!\!\!\!\!\!\!\!\!\!\!\!\!\!\!\!\!\!\!\!\!\!\!\!\!\!\!\!
      \Phi (t)\Phi ^{\prime \prime}(t) - (1+\nu)(\Phi^{\prime}(t))^{2}   \\
        \!\!\!\!\!\!\!\!\!\!\!\!\!\!\!\!\!\!\!\!\!\!\!\!\!\!
        &&\geq ~\Phi (t)[-2(1+2\nu )E(0)+2b+4(1+\nu )(\left\Vert P_{1}u_{1t}\right\Vert^{2}+\left\Vert P_{2}u_{2t}\right\Vert ^{2})] \\
        && ~~~~~~~~~
        -4(1+\nu )\Phi (t)(\left\Vert P_{1}u_{1t}\right\Vert ^{2}+\left\Vert P_{2}u_{2t}\right\Vert ^{2}+b) \\
        && =-2(1+2\nu )(E(0)+b)\Phi(t)~.
    \end{eqnarray*}
    If we choose $b\leq -E(0)$, then $\Phi (t)\Phi^{\prime \prime }(t)-(1+\nu)(\Phi^{\prime}(t))^{2}\geq 0$.
    Moreover
    \begin{equation*}
    \Phi'(0)=2\left\langle P_1\varphi_1,P_1\psi_1\right\rangle
        +2\left\langle P_2\varphi_2,P_2\psi_2\right\rangle+2bt_0\geq 0
    \end{equation*}
    for sufficiently large $t_{0}$. According to Lemma
    \ref{blow-up}, we observe that $\Phi (t)$ blows up in finite time. This  contradicts with the assumption of the
    existence of a global solution.
\end{proof}
\begin{remark}
The proof above implies that we may have blow-up even if $~E(0)>0$. In this case, all we need is to be able to choose
$b$ and $t_{0}$ so that
$\Phi \left( 0\right) >0$ and $\Phi ^{\prime }\left( 0\right) >0$. To shorten the notation put
\begin{equation*}
    A=\left\langle P_{1}\varphi _{1},P_{1}\psi _{1}\right\rangle+\left\langle P_{2}
     \varphi _{2},P_{2}\psi _{2}\right\rangle,~~~~~
    B=\left\Vert P_{1}\varphi_{1}\right\Vert ^{2}+\left\Vert P_{2}\varphi_{2}\right\Vert^{2}.
\end{equation*}
When $E(0)>0$,  by choosing $b=-E\left( 0\right) $ we still get blow up if there is some $t_{0}$ so that
initial data satisfies
\begin{equation*}
    A-E(0) t_{0}>0, ~~~~~B-E(0) t_{0}^{2}>0.
\end{equation*}
When  $A>0$, taking  $t_{0}=0$  works. When  $ A \leq 0$, then $t_{0}$ must be
chosen negative. The two inequalities can be rewritten as
\begin{equation*}
  E(0)^{-2}A^{2}<t_{0}^{2},~~~~~   t_{0}^{2}<E(0)^{-1}B.
\end{equation*}
Such a $t_{0}$ exists if and only if $A^{2}<E(0)B$. Hence there is blow-up if the initial data satisfies
\begin{equation*}
    \left( \left\langle P_{1}\varphi_{1},P_{1}\psi _{1}\right\rangle
    +\left\langle P_{2}\varphi_{2},P_{2}\psi_{2}\right\rangle\right) ^{2}<E\left( 0\right)
    \left( \left\Vert P_{1}\varphi_{1}\right\Vert ^{2}+\left\Vert P_{2}\varphi_{2}\right\Vert^{2}\right).
\end{equation*}

\end{remark}

\setcounter{equation}{0}
\section{Global Existence}
\noindent
Below we prove  global existence of solutions of (\ref{cp1})-(\ref{iv2}) for two different classes of kernel functions.
We note that  the kernel functions corresponding to these two particular cases belong to the classes of kernel functions
mentioned in Remark \ref{extension}. Thus, in the cases below, the local existence result of Theorem \ref{theo4}, and
hence Theorems \ref{global1} and \ref{global2}  can be extended to $s\geq 0$ for initial data in $H^{s}\cap L^{\infty}$.

\subsection{Sufficiently Smooth Kernels: $r_{1}, r_{2}>3$}
\noindent
We will now consider kernels $\beta_{i}$  $(i=1,2)$ that satisfy the estimate
$0 \leq  \widehat{\beta_{i}}(\xi) \leq C_{i}(1+\xi ^{2})^{-r_{i}/2}$ with $r_{i}>3$.
Typically if $\beta_{i}$ belongs to the Sobolev space  $W^{3,1}( {\Bbb R}) $ (i.e. $\beta_{i} $ and its derivatives  up to third-order are in $L^{1}$);
then we would get the estimate with $r_{i}=3$; hence we consider kernels that are slightly smoother than
those in $W^{3,1}({\Bbb R})$.

\begin{theorem}\label{global1}
    Let $~s>1/2~$, $~r_{1}, r_{2}> 3$. Let $\varphi _{i},\psi _{i}\in H^{s}$, $P_{i}\psi _{i}\in L^{2}$
    $(i=1,2) $ and $G(\varphi_{1},\varphi_{2})\in L^{1}$. If there is some $k>0$ so that
    $G(a,b)\geq -k(a^{2}+b^{2})$ for all $a,b\in {\Bbb R}$,
    then the Cauchy problem (\ref{cp1})-(\ref{iv2}) has a global solution with $u_{1}$ and $u_{2}$ in
    $C^{2}(\left[ 0,\infty \right) ,H^{s})$.
\end{theorem}
\begin{proof}
    Since $~r_{1}, r_{2}> 3$, by Theorem \ref{theo4} we have local existence.  The hypothesis implies
    that $E( 0) <\infty $. Assume that $u_{1}, u_{2}$ exist on $\left[ 0,T\right)$ for some $T>0$.
    Since $G( u_{1},u_{2}) \geq -k(u_{1}^{2}+u_{2}^{2})$, we get
    for all $t\in \left[ 0,T\right) $
    \begin{equation}
        \left\Vert P_{1}u_{1t}( t) \right\Vert ^{2}+\left\Vert P_{2}u_{2t}( t) \right\Vert^{2}\leq E(0)
            +(2k-1)(\left\Vert u_{1}( t) \right\Vert ^{2}+\left\Vert u_{2}(t)\right\Vert^{2}).  \label{ea}
    \end{equation}
    Noting that \ $\widehat{\beta _{i}}(\xi )\leq C_{i}(1+\xi ^{2})^{-r_{i}/2}$ for $i=1,2$; we have
    \begin{eqnarray}
        \left\Vert P_{1}u_{1t}(t)\right\Vert ^{2}
        = \left\Vert \widehat{P_{1}u}_{1t}(t)\right\Vert ^{2}
        &=& \int \xi^{-2}( \widehat{\beta}_{1}(\xi))^{-1}( \widehat{u}_{1t}(\xi ,t))^{2}d\xi   \nonumber \\
        &\geq & C^{-1}_{1}\int \xi ^{-2}(1+\xi ^{2})^{r_{1}/2}( \widehat{u}_{1t}( \xi ,t))^{2}d\xi   \nonumber \\
        &\geq & C^{-1}_{1}\int (1+\xi ^{2})^{( r_{1}-2) /2}(\widehat{u}_{1t}( \xi ,t) ) ^{2}d\xi   \nonumber \\
        & = & C^{-1}_{1}\left\Vert u_{1t}(t)\right\Vert_{\rho_{1}}^{2}~, \label{pa}
    \end{eqnarray}
    and similarly,
    \begin{equation}
        \left\Vert P_{2}u_{2t}(t)\right\Vert ^{2}\geq C^{-1}_{2}\left\Vert u_{2t}(t)\right\Vert _{\rho_{2}}^{2}  \label{pa2}
    \end{equation}
    where  $\rho_{i}=\frac{r_{i}}{2}-1$, $~i=1,2$.
    By the triangle inequality, for any Banach space valued differentiable function $w$ we have
    \begin{equation*}
        \frac{d}{dt}\left\Vert w( t) \right\Vert \leq \left\Vert \frac{d}{dt}w(t) \right\Vert ~.
    \end{equation*}
    Combining (\ref{ea}), (\ref{pa}) and (\ref{pa2}),
    \begin{eqnarray*}
        \frac{d}{dt}(\left\Vert u_{1}(t)\right\Vert_{\rho_{1}}^{2}&+&\left\Vert u_{2}(t)\right\Vert_{\rho_{2}}^{2}) \\
        &=& 2(\left\Vert u_{1}( t) \right\Vert _{\rho_{1}}\frac{d}{dt}\left\Vert u_{1}( t) \right\Vert _{\rho_{1}}
            +\left\Vert u_{2}(t)\right\Vert_{\rho_{2}}\frac{d}{dt}\left\Vert u_{2}(t)\right\Vert_{\rho_{2}}) \\
        &\leq & 2(\left\Vert u_{1t}( t) \right\Vert _{\rho_{1}}\left\Vert u_{1}( t) \right\Vert _{\rho_{1}}
            +\left\Vert u_{2t}( t) \right\Vert _{\rho_{2}}\left\Vert u_{2}( t) \right\Vert _{\rho_{2}}) \\
        &\leq & \left\Vert u_{1t}( t) \right\Vert _{\rho_{1}}^{2}+\left\Vert u_{1}( t) \right\Vert _{\rho_{1}}^{2}
            +\left\Vert u_{2t}( t) \right\Vert _{\rho_{2}}^{2}+\left\Vert u_{2}(t) \right\Vert _{\rho_{2}}^{2} \\
        &\leq & C(\left\Vert P_{1}u_{1t}( t) \right\Vert ^{2}+\left\Vert P_{2}u_{2t}(t) \right\Vert ^{2})
            +\left\Vert u_{1}( t)\right\Vert_{\rho_{1}}^{2}+\left\Vert u_{2}(t)\right\Vert _{\rho_{2}}^{2} \\
        &\leq & C[ E(0)+(2k-1)(\left\Vert u_{1}( t) \right\Vert^{2}+\left\Vert u_{2}(t) \right\Vert ^{2})]
        + \left\Vert u_{1}( t) \right\Vert_{\rho_{1}}^{2}+\left\Vert u_{2}(t) \right\Vert_{\rho_{2}}^{2} \\
        &\leq & CE(0)+(C( 2k-1)+1) (\left\Vert u_{1}( t)\right\Vert _{\rho_{1}}^{2}
        +\left\Vert u_{2}(t)\right\Vert _{\rho_{2}}^{2})
    \end{eqnarray*}
    where $C=\max (C_{1},C_{2})$.
    Gronwall's lemma implies that $\left\Vert u_{1}( t) \right\Vert_{\rho_{1}}+\left\Vert u_{2}( t) \right\Vert _{\rho_{2}}$
    stays bounded in $\left[ 0,T\right) $. Since $\rho_{i}=\frac{r_{i}}{2}-1>\frac{1}{2}$,
    $\left\Vert u_{1}(t)\right\Vert _{\infty}+\left\Vert u_{2}(t)\right\Vert _{\infty }$ also stays bounded in $\left[ 0,T\right) $.
     By Lemma \ref{lem9}, a global solution exists.
\end{proof}

\subsection{Kernels with Singularity}
\noindent
In the next theorem we will consider kernels of the form
$\beta_{1}(x)=\beta_{2}(x)=\gamma(\left\vert x\right\vert ) $ where
$\gamma\in C^{2}(\left[ 0,\infty \right) )$, $\gamma(0)>0$, $\gamma^{\prime }(0)<0$
and $\gamma^{\prime \prime }\in L^{1}\cap L^{\infty }$.
Then the $\beta_{i}$ will have a jump in the first derivative. The typical
example we have in mind is the Green's function
$\frac{1}{2}e^{-\left\vert x\right\vert }$. For such kernels we have
\begin{equation*}
\widehat{\beta }_{i}\left( \xi \right) \leq C_{i}\left( 1+\xi ^{2}\right)^{-1}
\end{equation*}
so $r_{1}=r_{2}=2.$ Due to the jump in $\beta_{i}^{\prime}$ at $x=0$, the distributional derivative will satisfy
\begin{equation*}
\beta_{i}^{\prime\prime}=\gamma^{\prime \prime }+2\gamma^{\prime}(0)\delta,~~~~i=1,2
\end{equation*}
where $\delta$ is the Dirac measure and we use the notation
$\gamma^{\prime \prime }( x) =\gamma^{\prime\prime } (\left\vert x\right\vert )$.
Then we have
\begin{equation*}
(\beta _{i}\ast w)_{xx}=\gamma^{\prime \prime }\ast w-\lambda w,~~~~i=1,2
\end{equation*}
where $\lambda=-2\gamma^{\prime }(0)>0$.
We will call these type of kernels \textit{mildly singular}. For such kernels we extend the global
existence result in \cite{constantin3} to the coupled system.
\begin{theorem}\label{global2}
    Let $s>1/2$ and let the kernels $\beta_{1}=\beta_{2}$ be mildly singular as defined above.
    Suppose that
    $\varphi_{1},\varphi_{2},\psi_{1},\psi_{2}\in H^{s}$, $P_{1}\psi_{1},P_{2}\psi_{2}\in L^{2}$
    and
    $G(\varphi_{1},\varphi_{2})\in L^{1}$. If there is some $C>0$, $k \geq 0$ and $q_{i}>1$ so that
    \begin{equation*}
    \left\vert g_{i}(a,b) \right\vert ^{q_{i}}\leq C [G(a,b)+k(a^{2}+b^{2})]
    \end{equation*}
    for all $a,b\in \mathbb{R}$ and $i=1,2$;
    then the Cauchy problem (\ref{cp1})-(\ref{iv2}) has a global solution with $u_{1}$ and $u_{2}$  in
    $C^{2}(\left[ 0,\infty\right),H^{s})$.
\end{theorem}
\begin{proof}
    By Theorem \ref{theo4} we have a local solution. Suppose the solution $(u_{1},u_{2})$  exists for
    $t\in \left[0,T\right)$. For fixed $x\in \mathbb{R}$ we define
    \begin{equation*}
     e( t) =\frac{1}{2}[(u_{1t}(x,t))^{2}+(u_{2t}(x,t))^{2}]+ \frac{\lambda}{2}[(u_{1}(x,t))^{2}
        +(u_{2}(x,t))^{2}+2G(u_{1}(x,t),u_{2}(x,t))] .
    \end{equation*}
    Then
    \begin{eqnarray*}
    e^{\prime }(t) &=& [u_{1tt}+\lambda ( u_{1}+g_{1}(u_{1},u_{2}))] u_{1t}
        +[ u_{2tt}+\lambda (u_{2}+g_{2}( u_{1},u_{2}) ) ] u_{2t} \\
    &=& [(\beta _{1}\ast ( u_{1}+g_{1}(u_{1},u_{2})))_{xx}+\lambda ( u_{1}+g_{1}( u_{1},u_{2}))]u_{1t} \\
    &~& + ~[(\beta _{2}\ast ( u_{2}+g_{2}(u_{1},u_{2})))_{xx}+\lambda ( u_{2}+g_{2}( u_{1},u_{2}))] u_{2t} \\
    &=& (\gamma^{\prime \prime }\ast u_{1}) u_{1t}+(\gamma^{\prime \prime }\ast g_{1}(u_{1},u_{2})) u_{1t}
            +(\gamma^{\prime \prime }\ast u_{2}) u_{2t}+( \gamma^{\prime \prime }\ast g_{2}( u_{1},u_{2})) u_{2t} \\
   &\leq &  (u_{1t})^{2}+(u_{2t})^{2}
             +\frac{1}{2}( \left\Vert \gamma^{\prime \prime }\ast u_{1}\right\Vert _{\infty }^{2}
             +\left\Vert \gamma^{\prime \prime }\ast u_{2}\right\Vert_{\infty }^{2}) \\
   &~&  + ~\frac{1}{2}(\left\Vert \gamma^{\prime \prime }\ast g_{1}( u_{1},u_{2})\right\Vert _{\infty }^{2}
            +\left\Vert \gamma^{\prime \prime }\ast g_{2}(u_{1},u_{2}) \right\Vert _{\infty }^{2}).
    \end{eqnarray*}
    Since $\gamma^{\prime \prime }\in L^{1}\cap L^{\infty }$  we have $\gamma^{\prime \prime }\in L^{p}$ for all
    $p\geq 1$. By Young's inequality
    \begin{eqnarray*}
       e^{\prime }(t)
        & \leq & (u_{1t})^{2}+(u_{2t})^{2}+\frac{1}{2}  \left\Vert \gamma^{\prime \prime }\right\Vert ^{2}
          (\left\Vert u_{1}\right\Vert ^{2}+\left\Vert u_{2}\right\Vert ^{2}) \\
        &~& + ~ \frac{1}{2}\left\Vert \gamma^{\prime \prime }\right\Vert _{L^{p_{1}}}^{2}\left\Vert g_{1}( u_{1},u_{2})\right\Vert_{L^{q_{1}}}^{2}
            +\frac{1}{2}\left\Vert \gamma^{\prime \prime }\right\Vert_{L^{p_{2}}}^{2}\left\Vert g_{2}( u_{1},u_{2}) \right\Vert_{L^{q_{2}}}^{2} ,
    \end{eqnarray*}%
    where $1/p_{i}+1/q_{i}=1$ $(i=1,2)$. Now the  terms may be estimated as
    \begin{equation*}
    \left\Vert u_{1}\right\Vert^{2}+\left\Vert u_{2}\right\Vert^{2}\leq E(0)
    \end{equation*}
     and for $i=1,2$
    \begin{eqnarray*}
     \left\Vert g_{i}( u_{1},u_{2}) \right\Vert _{L^{q_{i}}}^{2}
        &=&\left( \int \left\vert g_{i}( u_{1},u_{2}) \right\vert^{q_{i}}dx\right)^{2/q_{i}} \\
     & \leq &\left( C\int [G(u_{1},u_{2})+k(a^{2}+b^{2})] dx\right)^{2/q_{i}}
        \leq [ C(1+k)E( 0)]^{2/q_{i}}
    \end{eqnarray*}%
    so that
    \begin{equation*}
        e^{\prime }( t) \leq D+2e( t)
    \end{equation*}
    for some constant $D$ depending on $\left\Vert \gamma^{\prime \prime }\right\Vert_{L^{p_{i}}}$,
    $\left\Vert \gamma^{\prime \prime }\right\Vert $ and $E( 0) $ $(i=1,2)$.
    This inequality holds for all $x\in \mathbb{R}$, $t\in \left[0,T\right)$. Gronwall's lemma then implies that
    $e( t) $ and thus $u_{1}(x,t) $ and $u_{2}( x,t) $ stay bounded. Thus by Lemma \ref{lem9} we have global solution.
\end{proof}
\vspace*{20pt}

\noindent
{\bf Acknowledgement}: This work has been supported by the Scientific and Technological Research Council of
Turkey (TUBITAK) under the project TBAG-110R002.

\end{document}